\documentclass{article}
\usepackage{amsmath, amsthm, amsfonts, amssymb, amscd, epsfig}

\newtheorem{defn}{Definition}[section]
\newtheorem{thm}{Theorem}[section]
\newtheorem{lemma}{Lemma}[section]
\newtheorem{prop}{Proposition}[section]
\newtheorem{cor}{Corollary}[section]

\title{On a  Closed Binding Curve of One-holed Torus}
\author{Weiqiang Wu\\
Department of Mathematics\\
University of Maryland, College Park\\
\texttt{waikng@math.umd.edu}}

\begin{document}
\maketitle

\begin{abstract}
    Given a closed binding curve $\gamma$ of a surface $\Sigma$, any equivalence class of marked complete hyperbolic structure can be decomposed into polygons(possibly with a puncture) with sides being hyperbolic geodesic segments. When $\Sigma$ is a one-holed torus and $\gamma = A^3 B^2$, we show that any equivalence class of marked complete hyperbolic structure gives rise to an equilateral bigon with a puncture and a hexagon with equal opposite sides. In particular, we give a new coordinates of the Fricke Space of the one-holed torus.
\end{abstract}

\newpage

\section{Introduction}
    The length function of a simple closed curve on the Teichm\"{u}ller Space $\mathcal{T}_{\Sigma}$ (Fricke Space $\mathcal{F}_{\Sigma}$) of a closed surface $\Sigma$ has been studied by S.Kerckhoff \cite{Ke80} and S.Wolpert \cite{Wo87} since 1980's, and it plays an important role in understanding the geometry of $\mathcal{T}_{\Sigma}$. One of the most important features of these length functions is that they are convex along \textit{Earthquake Paths} (or \textit{Weil-Petersson geodesics}) \cite{Ke80, Wo87}.

    \begin{defn}[Fricke Space]
        $\mathcal{F}_{\Sigma} := \left\{ (f, X) \mid \Sigma \xrightarrow{f} X \mbox{ is a diffeomorphism}\right\} / \sim$, where $X$ is a complete hyperbolic surface and $(f,X) \sim (g, Y)$ if there exists a hyperbolic isometry $i: X \rightarrow Y$ such that the following diagram commutes up to isotopy.
        $$ \begin{CD}
        \Sigma @>f>> X\\
        @| @ViVV\\
        \Sigma @>g>>Y
        \end{CD}$$
    \end{defn}

    \begin{defn}[Closed Binding Curves]
        Let $\Sigma$ be a surface (possibly with punctures), a closed curve $\gamma$ is binding if $\gamma$ intersects itself in a minimal position and $\Sigma - \gamma$ is a union of disjoint disks (possibly with a puncture).
    \end{defn}

    Let $\gamma$ be a closed binding curve of $\Sigma$, the length function $l_{\gamma}: \mathcal{F}_{\Sigma} \rightarrow \mathbb{R}$ is not only \textit{strictly convex} along \textit{Earthquake Paths} (or \textit{Weil-Petersson geodesics}) but also \textit{proper}. Consequently, it has a \textit{unique} minimum at some marked hyperbolic structure $[f_{\gamma}, X_{\gamma}] \in \mathcal{F}_{\Sigma}$.

    Also, for any $[f, X] \in \mathcal{F}_{\Sigma}$, the unique closed hyperbolic geodesic isotopic to $f(\gamma)$ will cut $X$ into polygons(possibly with a puncture) with sides being hyperbolic geodesics, and conversely given these polygons $[f, X]$ can be reconstructed. Therefore these polygons may provide invariants and combinatoric ways to study $\mathcal{F}_{\Sigma}$.

    \begin{figure}[h]
       \begin{center}
           \scalebox{.6}{\epsffile{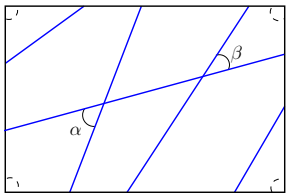}}
       \end{center}
       \caption{The closed binding curve $A^3 B^2$ in the one-holed torus.}
       \label{fig:A3B2}
    \end{figure}

    When $\Sigma_{1,1}$ is the one-holed torus and $\gamma = A^3 B^2$ as shown in Figure \ref{fig:A3B2}, any $[f,X] \in \mathcal{F}_{\Sigma_{1,1}}$ can be decomposed into
    \begin{itemize}
        \item A bigon with side lengths $a$ and $b$ and angles $\alpha$ and $\beta$, which forms a cusp region.
        \item A Hexagon with side lengths $a$, $c$, $d$, $b$, $c'$ and $d'$ and angles $\pi-\beta$, $\alpha$, $\pi-\beta$, $\pi-\alpha$, $\beta$ and $\pi - \alpha$ such that $c = c'$ and $d = d'$.
    \end{itemize}
    as in Figure \ref{fig:POLYGONS}.

    \begin{figure}[h]
       \begin{center}
           \scalebox{.6}{\epsffile{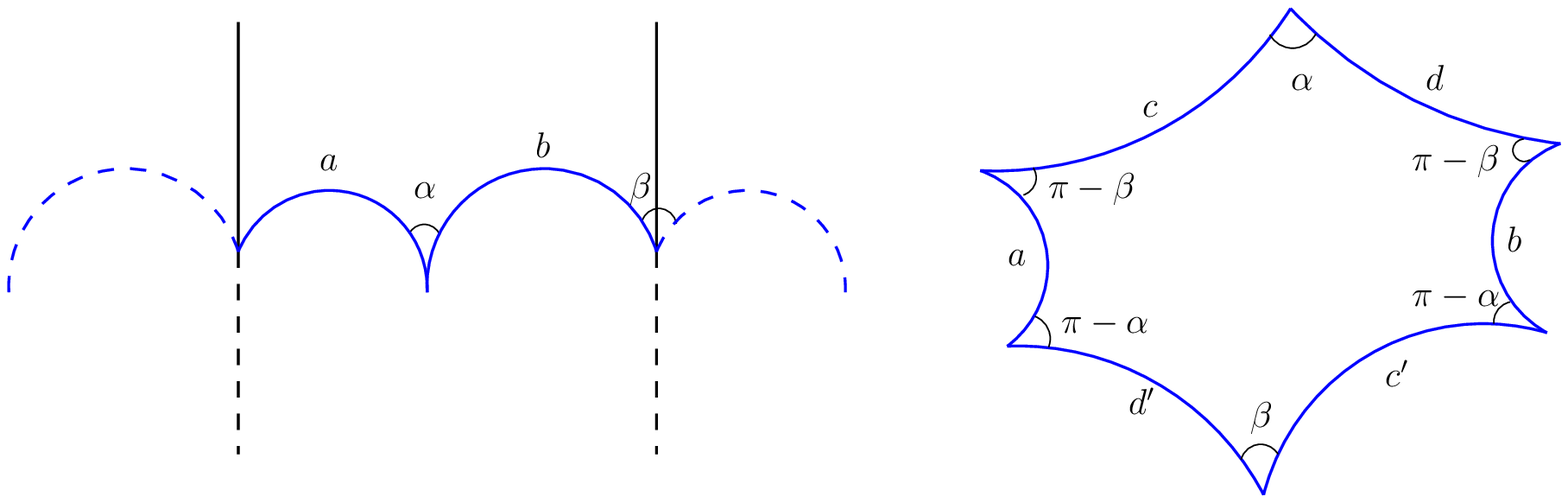}}
       \end{center}
       \caption{The punctured bigon and hexagon in a decomposition.}
       \label{fig:POLYGONS}
    \end{figure}

    Let $\mathcal{P} := \left\{ \mbox{ compatible pairs of a punctured bigon and a hexagon} \right\}$, then there is a $1$-$1$ correspondence
    $$\mathcal{P} \mathop{\rightleftharpoons}^{g}_{c} \mathcal{F}_{\Sigma_{1,1}}$$ which comes from gluing and cutting. Let $\mathcal{P}_0 \subset \mathcal{P}$ be the subset consisting of those with $\alpha = \beta$ and $a = b$.

    \begin{defn}[Length Function on $\mathcal{P}$]
        The length function $l: \mathcal{P} \rightarrow \mathbb{R}$ is given by $a+b+c+d$.
    \end{defn}

    Therefore the diagram
     $$ \begin{CD}
        \mathcal{P} @>g>> \mathcal{F}_{\Sigma_{1,1}}\\
        @VlVV @Vl_{\gamma}VV\\
        \mathbb{R} @= \mathbb{R}
     \end{CD}$$
    commutes by construction of $l$.

\section{Main Theorems}

    \begin{thm}
        $\mathcal{P}_0 = \mathcal{P}$. In particular, for any $[f,X] \in \mathcal{F}_{\Sigma_{1,1}}$, $$\alpha = \beta \in \left(0, \frac{2\pi}{3}\right),$$ and $$a = b = \log\left(\frac{1+\cos\left(\frac{\alpha}{2}\right)}{1-\cos\left(\frac{\alpha}{2}\right)}\right).$$
        \label{thm:P0}
    \end{thm}

    \begin{figure}[h]
       \begin{center}
           \scalebox{.5}{\epsffile{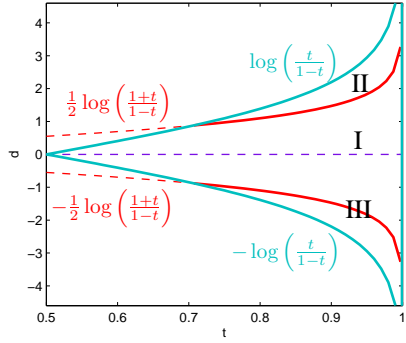}}
       \end{center}
       \caption{The region of $\mathcal{V}$.}
       \label{fig:REGION}
    \end{figure}

    \newpage
    First, we need the following lemmas.
    \begin{lemma}[Existence]
        There exists an injective map $j: \mathcal{V} \rightarrow \mathcal{P}_0$, where
        $$ \mathcal{V} := \left\{ (t, s)\in \mathbb{R}^2 \mid  -\log\left(\frac{t}{1-t}\right)< s < \log\left(\frac{t}{1-t}\right), \frac{1}{2}< t < 1\right\}.$$
        \label{thm:EXISTENCE}
    \end{lemma}
    \vspace{-0.25 in}
    This lemma will be proved in Section \ref{sec:EXISTENCE}.

    \noindent \textbf{Remark}
    \begin{itemize}
        \item $s$ will be the displacement of the mid point of $a$ from the common perpendicular of $a$ and $b$.
        \item $t$ will be $\cos\left(\frac{\alpha}{2}\right)$.
    \end{itemize}

    \begin{lemma}[Properness]
        The pullback of the length function $j^*(l): \mathcal{V} \rightarrow \mathbb{R}$ is proper.
        \label{thm:PROPERNESS}
    \end{lemma}
        This lemma will be proved in Section \ref{sec:PROPERNESS}.

    \begin{thm}
        $g\circ j: \mathcal{V} \rightarrow \mathcal{F}_{\Sigma_{1,1}}$ is a diffeomorphism.
        \label{thm:MAIN}
    \end{thm}

    \begin{proof}[Proof of Theorem \ref{thm:MAIN}]
        Note that $\mathcal{F}_{\Sigma_{1,1}}$ is diffeomorphic to $\mathbb{R}^2$, hence $g \circ j$ is an injective local diffeomorphism.

        In addition $g \circ j$ is proper, since $j^*(l)$ and $l_{\gamma}$ are proper by Lemma \ref{thm:PROPERNESS} and the following diagram commutes.
        $$ \begin{CD}
        \mathcal{V} @>g \circ j>> \mathcal{F}_{\Sigma_{1,1}}\\
        @Vj^*(l)VV @Vl_{\gamma}VV\\
        \mathbb{R} @= \mathbb{R}
        \end{CD}$$

        Therefore $g \circ j$ is a diffeomorphism by \textit{Invariance of Domain} (See \cite{Ha}).
    \end{proof}

    \newpage

    \begin{proof}[Proof of Theorem \ref{thm:P0}]
        $g \circ j$ is onto by Theorem \ref{thm:MAIN}, hence $P_0 = P$. Since $\cos\left(\frac{\alpha}{2}\right) = t \in \left(\frac{1}{2}, 1\right)$ (see remark following lemma \ref{thm:EXISTENCE}),
        $$ \alpha = \beta \in \left(0, \frac{2\pi}{3}\right)$$
        and
        $$ a = b = \log\left(\frac{1+\cos\left(\frac{\alpha}{2}\right)}{1-\cos\left(\frac{\alpha}{2}\right)}\right)$$
        follows from the following lemma.
    \end{proof}

    \begin{lemma}
        If $a=b$ and $\alpha = \beta \in (0, \pi)$, then
        $$ a = \log\left(\frac{1+\cos\left(\frac{\alpha}{2}\right)}{1-\cos\left(\frac{\alpha}{2}\right)}\right).$$
        \label{thm:CUSP}
    \end{lemma}

    \begin{proof}
        Use the upper half plane model with the hyperbolic metric $$ ds^2 = \frac{dx^2+dy^2}{y^2}.$$
        The length of the geodesic segment $a$ is given by
        $$ a = \int_{\frac{\alpha}{2}}^{\pi-\frac{\alpha}{2}} \frac{d \theta}{ \sin(\theta) }= \log\left(\frac{1+\cos\left(\frac{\alpha}{2}\right)}{1-\cos\left(\frac{\alpha}{2}\right)}\right).$$
    \end{proof}

    \begin{figure}[h]
       \begin{center}
           \scalebox{.6}{\epsffile{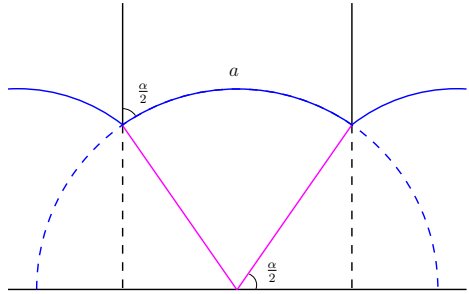}}
       \end{center}
       \caption{Relation between $a$ and $\alpha$ in the cusp.}
       \label{fig:CUSP}
    \end{figure}

\section{Existence of Compatible Pairs from $\mathcal{V}$}
\label{sec:EXISTENCE}
    In this section, we will prove \textbf{Lemma \ref{thm:EXISTENCE}}, i.e.

    \bigskip
    \textit{There exists an injective map $j: \mathcal{V} \rightarrow \mathcal{P}_0$, where}
    \begin{equation} \mathcal{V} := \left\{ (t, s)\in \mathbb{R}^2 \mid  -\log\left(\frac{t}{1-t}\right)< s < \log\left(\frac{t}{1-t}\right), \frac{1}{2}< t < 1\right\}. \label{eqn:V}\end{equation}
    \smallskip

    Given $t \in \left(\frac{1}{2}, 1\right)$, Lemma \ref{thm:CUSP} also guarantees the existence of the punctured bigon with
    \begin{eqnarray}
    && \alpha = \beta = 2 \cos^{-1}(t) \notag\\
    && a =  b = \log\left(\frac{1+t}{1-t}\right). \label{eqn:a}
    \end{eqnarray}

    We divide $\mathcal{V}$ into three parts as also shown in Figure \ref{fig:REGION}.
    \begin{itemize}
        \item $I = \left\{ (t,s) \in \mathcal{V} \mid -\frac{1}{2}\log\left(\frac{1+t}{1-t}\right) < s< \frac{1}{2}\log\left(\frac{1+t}{1-t}\right)\right\}$
        \item $II = \left\{ (t,s) \in \mathcal{V} \mid  s < -\frac{1}{2}\log\left(\frac{1+t}{1-t}\right) \mbox{ or } s > \frac{1}{2}\log\left(\frac{1+t}{1-t}\right)\right\}$
        \item $III = \left\{ (t,s) \in \mathcal{V} \mid s = \pm \frac{1}{2}\log\left(\frac{1+t}{1-t}\right)\right\}$
    \end{itemize}

    Injectivity of $j$ will follow from our construction.

\subsection{Existence of Type I Hexagons}
    For any $(t,s)$ in
    \begin{equation} I = \left\{ (t,s) \in \mathcal{V} \mid -\frac{1}{2}\log\left(\frac{1+t}{1-t}\right) < s < \frac{1}{2}\log\left(\frac{1+t}{1-t}\right)\right\} \label{eqn:TYPE1}\end{equation}
    we are going to construct the hexagon as shown in Figure \ref{fig:TYPE1}.
    \begin{figure}[h]
       \begin{center}
           \scalebox{.6}{\epsffile{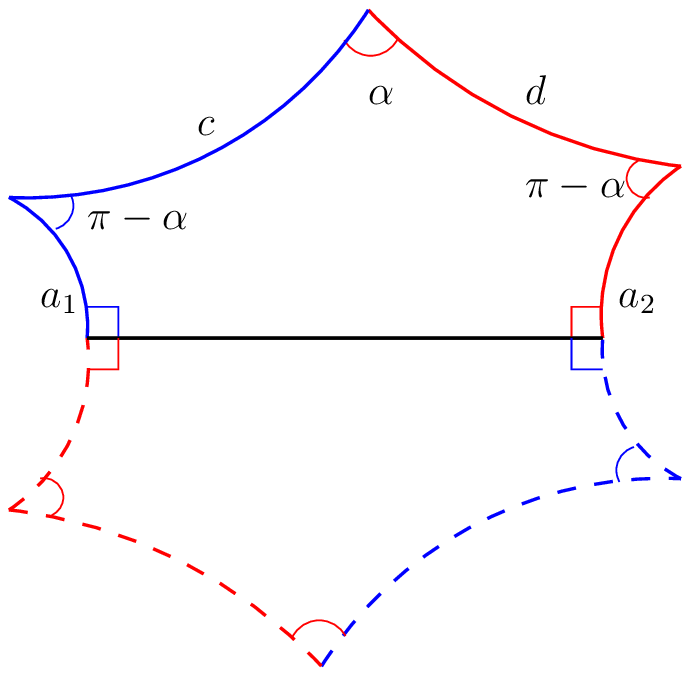}}
       \end{center}
       \caption{Shape of a Type I Hexagon.}
       \label{fig:TYPE1}
    \end{figure}

    \begin{itemize}
    \item
    First of all, $\alpha$ and $a$ is determined by (\ref{eqn:a}) and let
    \begin{eqnarray}
        && a_1 := \frac{a}{2} + s = \frac{1}{2} \log\left(\frac{1+t}{1-t}\right) + s \notag\\
        && a_2 := \frac{a}{2} - s = \frac{1}{2} \log\left(\frac{1+t}{1-t}\right) - s.
    \end{eqnarray}

    Therefore by (\ref{eqn:V}) and (\ref{eqn:TYPE1}),
    \begin{equation} \label{eqn:TYPE1a1a2}
        a_1, a_2 \in \left\{
        \begin{array}{lcl}
                \left(\frac{1}{2}\log\left(\frac{1+t}{1-t}\right) - \log\left(\frac{t}{1-t}\right), \frac{1}{2}\log\left(\frac{1+t}{1-t}\right) + \log\left(\frac{t}{1-t}\right)\right) & , & t \in \left(\frac{1}{2}, \frac{\sqrt{2}}{2}\right)\\ \\
                \left(0, \log\left(\frac{1+t}{1-t}\right)\right) & , & t \in \left[\frac{1}{2}, 1\right)
            \end{array}\right.
    \end{equation}

    \item Second, we need to show
    \begin{prop}\label{thm:TYPE1C}
        The geodesic ray along $c$ does not intersect the common perpendicular on the right; and the geodesic ray along $d$ does not intersect the common perpendicular on the left either. (see Figure \ref{fig:TYPE1})
    \end{prop}

    Recall the following fact from hyperbolic geometry (See \cite{Ra}).
    \begin{lemma}\label{thm:AREA}
        The area of a hyperbolic triangle (possibly with ideal vertices) is given by $$\pi - \alpha - \beta -\gamma$$ where $\alpha, \beta, \gamma$ are the inner angles. In particular, $\alpha + \beta + \gamma < \pi$.
    \end{lemma}

    \begin{figure}[h]
       \begin{center}
           \scalebox{.6}{\epsffile{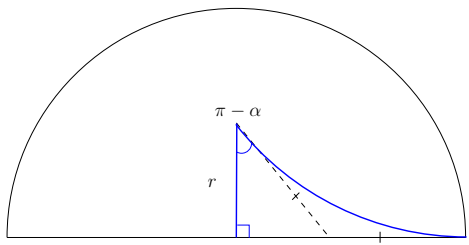}}
       \end{center}
       \caption{Non-intersecting Condition.}
       \label{fig:TYPE1C}
    \end{figure}

    \begin{proof}[Proof of Proposition \ref{thm:TYPE1C}]
        When $t \in \left[\frac{\sqrt{2}}{2}, 1\right)$ (i.e. $\alpha \in \left(0,  \frac{\pi}{2}\right]$), the proposition is obviously true from Lemma \ref{thm:AREA}. Therefore we may assume $t \in \left(\frac{1}{2}, \frac{\sqrt{2}}{2}\right)$ (i.e. $\alpha \in \left(\frac{\pi}{2}, \frac{2\pi}{3}\right)$).

        From Figure \ref{fig:TYPE1C}, the Euclidean length of $r$, which is the least distance to keep away from intersecting, is given by $$ \frac{\cos(\pi-\alpha)}{1+\sin(\pi-\alpha)} = - \frac{\cos(\alpha)}{1+\sin(\alpha)}.$$
        Therefore using the Poincar\'{e} disk model with the hyperbolic metric $$ds^2 = 4 \frac{dx^2+dy^2}{\left(1-(x^2+y^2)\right)^2},$$
        \begin{eqnarray}
            r & = & \int_{0}^{- \frac{\cos(\alpha)}{1+\sin(\alpha)}} \frac{2 dr}{1 - r^2} \notag \\
            & = & \log\left( \frac{1 + \sin(\alpha) - \cos(\alpha)}{1 + \sin(\alpha) + \cos(\alpha)} \right) \notag \\
            & = & \log\left( \frac{1 + 2 t \sqrt{1-t^2} - (2t^2 - 1)}{1 + 2 t \sqrt{1-t^2} + (2t^2 - 1)} \right) \notag\\
            & = & \frac{1}{2}\log(1-t) + \frac{1}{2}\log(1+t) - \log(t).
            \label{eqn:r}
        \end{eqnarray}
        and hence $a_1 > r$ and $a_2 > r$ from (\ref{eqn:TYPE1a1a2}). Therefore the proposition is true as well.
    \end{proof}

    \item
    And last, we are ready to show the existence of Type I hexagons.

    \begin{figure}[h]
       \begin{center}
           \scalebox{.6}{\epsffile{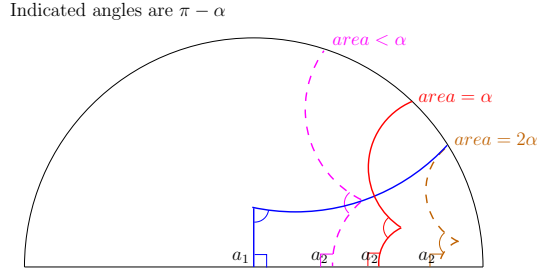}}
       \end{center}
       \caption{Existence of Type I Hexagons.}
       \label{fig:TYPE1E}
    \end{figure}

    \begin{proof}
        Without loss of generality, we may assume $s \leq 0$, then $a_1 \leq a_2$.

        See Figure \ref{fig:TYPE1E}. Since the \textit{blue line} does not intersect the \textit{horizonal line} by Proposition \ref{thm:TYPE1C} and  $a_1 \leq a_2$, there is a \textit{magenta line} such that its vertex is on the \textit{blue line}. They form a quadrilateral with angles $\pi-\alpha, \frac{\pi}{2}, \frac{\pi}{2}$ and some nonzero angle, therefore its area is less than $\alpha$. (When $s = 0$, it is a degenerate quadrilateral with area $0 < \alpha$)

        Continue parallel translating the \textit{magenta line} to the right, then there is a \textit{brown line} such that it intersects the \textit{blue line} at infinity. They form a pentagon with angles $\pi-\alpha, \frac{\pi}{2}, \frac{\pi}{2}, \pi-\alpha$ and $0$, therefore its area is equal to $2\alpha$.

        Therefore there is a unique \textit{red line} between the \textit{magenta line} and the \textit{brown line}, such that the area bounded by it together with the \textit{blue line} is exactly $\alpha$, hence the angle between the \textit{blue line} and the \textit{red line} is $\alpha$.

        By doubling the pentagon formed by the \textit{blue line} and the \textit{red line}, we find the desired Type I hexagon.
    \end{proof}

    \end{itemize}
\subsection{Existence of Type II Hexagons}
    For any $(t,s)$ in
    \begin{equation} II = \left\{ (t,s) \in \mathcal{V} \mid  s < -\frac{1}{2}\log\left(\frac{1+t}{1-t}\right) \mbox{ or } s > \frac{1}{2}\log\left(\frac{1+t}{1-t}\right)\right\}. \label{eqn:TYPE2}\end{equation}

    From (\ref{eqn:V}) and (\ref{eqn:TYPE2}), $t \in \left(\frac{\sqrt{2}}{2}, 1\right)$ (i.e. $\alpha \in \left(0, \frac{\pi}{2}\right)$. And without loss of generality, we assume $s<0$ in the following.

    We are going to construct the hexagon as shown in Figure \ref{fig:TYPE2}.
    \begin{figure}[h]
       \begin{center}
           \scalebox{.6}{\epsffile{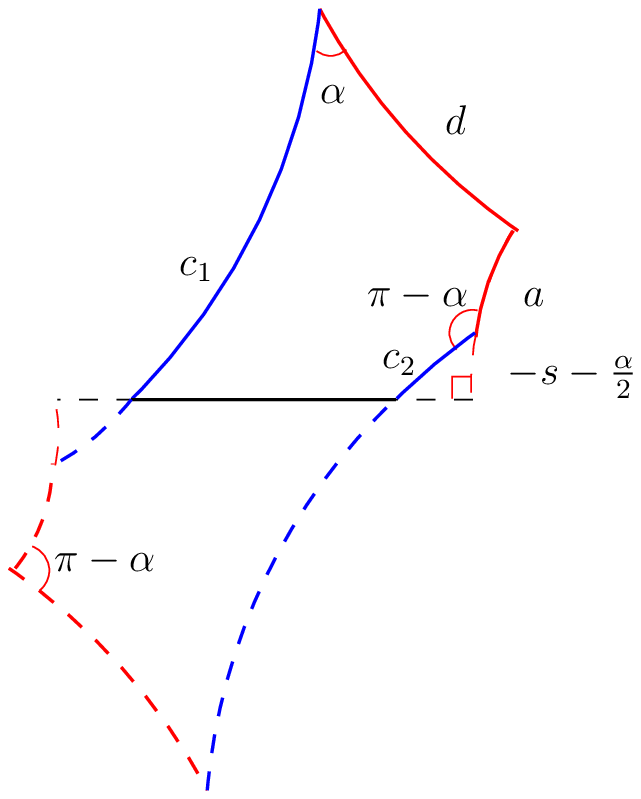}}
       \end{center}
       \caption{Shape of a Type II Hexagon.}
       \label{fig:TYPE2}
    \end{figure}

    \begin{itemize}
        \item
        First of all, we need to construct the right triangle at the right corner of Figure \ref{fig:TYPE2}.
        \begin{prop}
            There exists a right triangle with a side of length $-s-\frac{a}{2}$ and the other adjacent angle being $\alpha$.
        \end{prop}

        \begin{proof}
            Since $$ -\log\left(\frac{t}{1-t}\right) < s < -\frac{1}{2}\log\left(\frac{1+t}{1-t}\right)$$ from (\ref{eqn:V}) and (\ref{eqn:TYPE2}) and $$a = \log\left(\frac{1+t}{1-t}\right),$$ from (\ref{eqn:a}),
            \begin{equation}\label{eqn:TYPE2d}
                -s - \frac{a}{2} \in \left(0, \log\left(\frac{t}{1-t}\right) - \frac{1}{2}\log\left(\frac{1+t}{1-t}\right) \right).
            \end{equation}

            Use the same picture as in Figure \ref{fig:TYPE1C} but replacing the angle $\pi-\alpha$ with $\alpha$, the least distance from being intersecting is given by
            \begin{eqnarray}
                r' & = & \log\left( \frac{1 + \sin(\alpha) + \cos(\alpha)}{1 + \sin(\alpha) - \cos(\alpha)} \right) \notag \\
                & = & \log\left( \frac{1 + 2 t \sqrt{1-t^2} + (2t^2 - 1)}{1 + 2 t \sqrt{1-t^2} - (2t^2 - 1)} \right) \notag\\
                & = & \log(t) - \frac{1}{2}\log(1-t) - \frac{1}{2}\log(1+t).
                \label{eqn:r'}
            \end{eqnarray}
            Then $-s - \frac{a}{2} < r'$ by (\ref{eqn:TYPE2d}), hence there exists a unique such right triangle.
        \end{proof}

    \item
        Note in Figure \ref{fig:TYPE2} that $c_1$ is parallel to $c_2$ along the common perpendicular and the geodesic along $d$ does not intersect  the common perpendicular on the left since $\pi - \alpha > \frac{\pi}{2}$. Then we are ready to show the existence of Type II hexagons.

        \begin{figure}[h]
           \begin{center}
               \scalebox{.6}{\epsffile{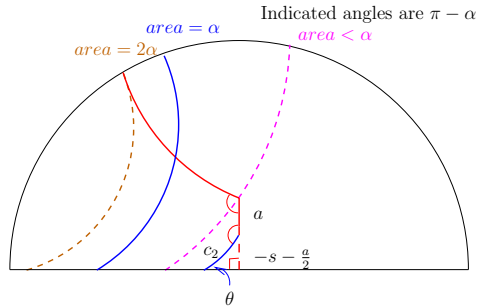}}
           \end{center}
           \caption{Existence of Type II Hexagons.}
           \label{fig:TYPE2E}
        \end{figure}

        \begin{proof}
            See Figure \ref{fig:TYPE2E}. Parallel translating $c_2$ along the \textit{horizontal line} to the left, there is a \textit{magenta line} such that it meets the \textit{red line} at the vertex, together with $c_2$ and the \textit{horizontal line} they form a quadrilateral with angles $\theta, \pi-\theta, \pi-\alpha$ and some nonzero angle, therefore its area is less than $\alpha$.

            Continue parallel translating the \textit{magenta line} to the left, then there is a \textit{brown line} such that it intersects the \textit{red line} at infinity. They form a pentagon with angles $\theta, \pi-\theta, \pi-\alpha, \pi-\alpha$ and $0$, therefore its area is equal to $2\alpha$.

            Therefore there is a unique \textit{blue line} between the \textit{magenta line} and the \textit{brown line}, such that the area inscribed by it together with the \textit{red line} and $c_2$ is exactly $\alpha$, hence the angle between the \textit{blue line} and the \textit{red line} is $\alpha$.

            By doubling the pentagon formed by the \textit{blue line}, the \textit{red line} and $c_2$, we find the desired Type II hexagon.
        \end{proof}
    \end{itemize}

\subsection{Existence of Type III Hexagons}
    For any $(t,s)$ in
    \begin{equation} III = \left\{ (t,s) \in \mathcal{V} \mid  s = \pm \frac{1}{2}\log\left(\frac{1+t}{1-t}\right)\right\}. \label{eqn:TYPE3}\end{equation}

    From (\ref{eqn:V}) and (\ref{eqn:TYPE2}), $t \in \left(\frac{\sqrt{2}}{2}, 1\right)$ (i.e. $\alpha \in \left(0, \frac{\pi}{2}\right)$. And without loss of generality, we assume $s = -\frac{1}{2}\log\left(\frac{1+t}{1-t}\right)$ in the following.

    We are going to construct the hexagon as shown in Figure \ref{fig:TYPE3}.
    \begin{figure}[h]
       \begin{center}
           \scalebox{.6}{\epsffile{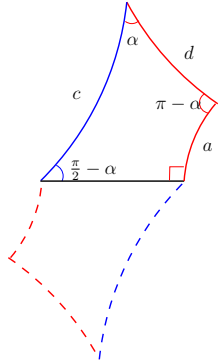}}
       \end{center}
       \caption{Shape of a Type III Hexagon.}
       \label{fig:TYPE3}
    \end{figure}

    \begin{itemize}
        \item
        First of all, like Proposition \ref{thm:TYPE1C} the geodesic ray along $d$ does not intersect the common perpendicular on the left either in this case, since $\pi - \alpha > \frac{\pi}{2}$.

        \item
        Then we are ready to show the existence of Type III hexagons.

        \begin{figure}[h]
           \begin{center}
               \scalebox{.6}{\epsffile{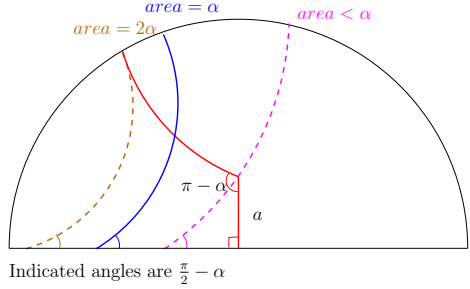}}
           \end{center}
           \caption{Existence of Type III Hexagons.}
           \label{fig:TYPE3E}
        \end{figure}

        \begin{proof}
            See Figure \ref{fig:TYPE3E}. There is a \textit{magenta line} such that it meets the \textit{red line} at the vertex, they form a triangle with angles $\frac{\pi}{2}-\alpha, \frac{\pi}{2}$ and some nonzero angle, therefore its area is less than $\alpha$.

            Continue parallel translating the \textit{magenta line} to the left, then there is a \textit{brown line} such that it intersects the \textit{red line} at infinity. They form a quadrilateral with angles $\frac{\pi}{2}-\alpha, \frac{\pi}{2}, \pi-\alpha$ and $0$, therefore its area is equal to $2\alpha$.

            Therefore there is a unique \textit{blue line} between the \textit{magenta line} and the \textit{brown line}, such that the area inscribed by it together with the \textit{red line} is exactly $\alpha$, hence the angle between the \textit{blue line} and the \textit{red line} is $\alpha$.

            By doubling the quadrilateral formed by the \textit{blue line} and the \textit{red line}, we find the desired Type III hexagon.
        \end{proof}
    \end{itemize}

    This concludes the proof of Lemma \ref{thm:EXISTENCE}.

\section{Properness of the Length Function}
\label{sec:PROPERNESS}
In this section, we will prove \textbf{Lemma \ref{thm:PROPERNESS}}, i.e.

\bigskip
\noindent \textit{The pullback of the length function $j^*(l): \mathcal{V} \rightarrow \mathbb{R}$ is proper, where $l$ is given by $a+b+c+d$.}
\bigskip

If suffices to show that if any sequence $\{(t_n, s_n)\} \subset \mathcal{V}$ leaves any compact set of $\mathcal{V}$, then $j^*(l)\left((t_n, s_n)\right) \rightarrow \infty$.

We may assume $\{t_n\}$ converges to $\hat{t} \in \left[\frac{1}{2}, 1\right]$. Then there are four cases to consider
\begin{itemize}
    \item $\hat{t} = 1$
    \item $\hat{t} \in \left( \frac{\sqrt{2}}{2}, 1\right)$
    \item $\hat{t} \in \left[ \frac{1}{2}, \frac{\sqrt{2}}{2}\right)$
    \item $\hat{t} = \frac{\sqrt{2}}{2}$
\end{itemize}

\subsection{Case: $\hat{t} = 1$}
\begin{proof}
    Since $t_n \rightarrow \hat{t} = 1$, from (\ref{eqn:a}) $$a_n = \log\left(\frac{1+t_n}{1-t_n}\right) \rightarrow \infty,$$
    and note that  $j^*(l)\left((t_n, s_n)\right) > a_n$ hence
    $$ j^*(l)\left((t_n, s_n)\right) \rightarrow \infty.$$
\end{proof}

\subsection{Case: $\hat{t} \in \left( \frac{\sqrt{2}}{2}, 1\right)$}
\begin{proof}
    In this case, without loss of generality we may assume $s_n < 0$, $\{(t_n, s_n)\} \subset II$ and $\{s_n\}$ converges to $\hat{s} = -\log\left(\frac{\hat{t}}{1-\hat{t}}\right)$.

    \begin{figure}[h]
        \begin{center}
            \scalebox{.6}{\epsffile{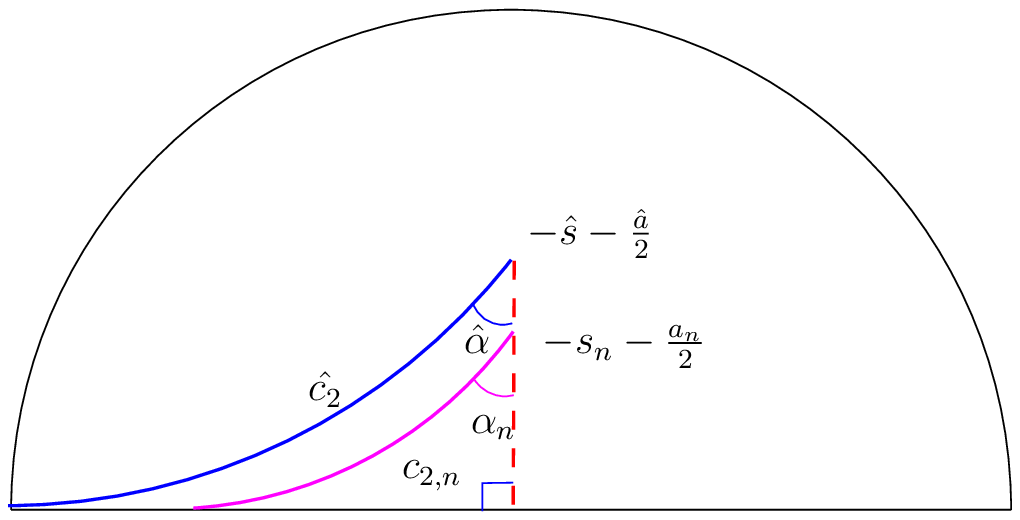}}
        \end{center}
        \caption{$t_n \rightarrow \hat{t} \in \left( \frac{\sqrt{2}}{2}, 1\right)$.}
        \label{fig:PROPER2}
    \end{figure}

    We consider again the right triangle as in Figure \ref{fig:TYPE2}. We lift all the geodesic segments $-s_n - \frac{a_n}{2}$'s on the \textit{vertical line} at the origin as in Figure \ref{fig:PROPER2}. Since
    $$ -\hat{s} - \frac{\hat{a}}{2} = \log\left(\frac{\hat{t}}{1-\hat{t}}\right) - \frac{1}{2}\log\left(\frac{1+\hat{t}}{1-\hat{t}}\right) = r' > 0$$ from (\ref{eqn:r'}).
    The geodesic along $\hat{c_2}$ is intersecting the \textit{horizontal line} at infinity, hence $\hat{c_2} = \infty$.

    Since $(t_n, s_n) \rightarrow (\hat{t}, \hat{s})$, $$ \alpha_n \rightarrow \hat{\alpha}$$ $$ -s_n - \frac{a_n}{2} \rightarrow -\hat{s} - \frac{\hat{a}}{2}$$ then $$ c_{2,n} \rightarrow \hat{c_2} = \infty.$$ Since  $j^*(l)\left((t_n, s_n)\right) > c_n > c_{2,n}$,
    $$ j^*(l)\left((t_n, s_n)\right) \rightarrow \infty.$$
\end{proof}

\subsection{Case: $\hat{t} \in \left[ \frac{1}{2}, \frac{\sqrt{2}}{2}\right)$ }
\begin{proof}
    In this case, without loss of generality we may assume $s_n \leq 0$, $\{(t_n, s_n)\} \subset I$ and $\{s_n\}$ converges to $\hat{s} = -\log\left(\frac{\hat{t}}{1-\hat{t}}\right)$.

    \begin{figure}[h]
        \begin{center}
            \scalebox{.6}{\epsffile{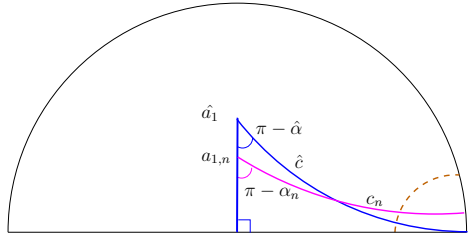}}
        \end{center}
        \caption{$t_n \rightarrow \hat{t} \in \left[ \frac{1}{2}, \frac{\sqrt{2}}{2}\right)$.}
        \label{fig:PROPER4}
    \end{figure}

    We lift all the geodesic segments $a_{1,n}$'s on the \textit{vertical line} at the origin as in Figure \ref{fig:PROPER4}. Since $$\hat{a_1} = \frac{\hat{a}}{2} + \hat{s} = \frac{1}{2}\log\left(\frac{1+\hat{t}}{1-\hat{t}}\right) - \log\left(\frac{\hat{t}}{1-\hat{t}}\right) = r > 0$$ from (\ref{eqn:TYPE1a1a2}) and (\ref{eqn:r}), the geodesic along $\hat{c}$ intersects the \textit{horizontal line} at infinity, hence $\hat{c} = \infty$.

    For any $N>0$, we can choose a \textit{brown line} geodesic, which is perpendicular to the \textit{horizontal line}, such that the distance between the \textit{vertical line} and the \textit{brown line} is greater than $N$.

    Since $(t_n, s_n) \rightarrow (\hat{t}, \hat{s})$,
    $$ \alpha_n \rightarrow \hat{\alpha}$$
    $$ a_{1,n} \rightarrow \hat{a_1}.$$
    Therefore the geodesic along $c_n$ has to intersect the \textit{brown line}, when $n$ is large enough.

    Note that for any Type II hexagon, $c$ and $a_2$ are parallel (See Figure \ref{fig:TYPE1}). Therefore the geodesic along $a_{2,n}$ has to be on the right hand side of the \textit{brown line}. Therefore $a_{1,n} + c_n + d_n + a_{2, n} > N$ by triangular inequality. Let $N \rightarrow \infty$, $$a_{1,n} + c_n + d_n + a_{2, n} \rightarrow \infty.$$

    Since $j^*(l)\left((t_n, s_n)\right) > a_n + c_n + d_n = a_{1,n} + c_n + d_n + a_{2, n}$,
    $$ j^*(l)\left((t_n, s_n)\right) \rightarrow \infty.$$
\end{proof}

\subsection{Case: $\hat{t} = \frac{\sqrt{2}}{2}$}
\begin{proof}
    In this case, without loss of generality we may assume $s_n < 0$, $\{(t_n, s_n)\}$ is either completely contained in $I$, $II$ or $III$ and $\{s_n\}$ converges to $$\hat{s} = -\log\left(\frac{\hat{t}}{1-\hat{t}}\right) = -\log\left(\sqrt{2}+1\right).$$

    From (\ref{eqn:a}) $$ \hat{\alpha} = 2 \cos^{-1}(\hat{t}) = \frac{\pi}{2}$$ $$ \hat{a} = \log\left(\frac{1+\hat{t}}{1-\hat{t}}\right) = 2 \log\left(\sqrt{2}+1\right).$$

    \begin{itemize}
        \item
            If $\{(t_n, s_n)\} \subset I$, from (\ref{eqn:TYPE1a1a2}) $$ \hat{a_1} = \frac{\hat{a}}{2} + \hat{s} = 0.$$ Therefore
            $$ \alpha_n \rightarrow \frac{\pi}{2}$$
            $$ a_{1,n} \rightarrow 0.$$

            \begin{figure}[h]
                \begin{center}
                    \scalebox{.6}{\epsffile{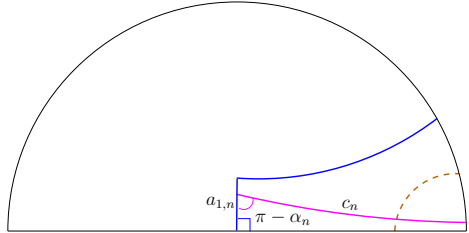}}
                \end{center}
                \caption{$\alpha_n \rightarrow \frac{\pi}{2}$ and $a_{1,n} \rightarrow 0$.}
                \label{fig:PROPER31}
            \end{figure}

            We lift all the geodesic segments $a_{1,n}$'s on the \textit{vertical line} at the origin as in Figure \ref{fig:PROPER31}. For any $N>0$, we can choose a \textit{brown line} geodesic, which is perpendicular to the \textit{horizontal line}, such that the distance between the \textit{vertical line} and the \textit{brown line} is greater than $N$.

            Then when $n$ is large enough, the geodesic along $c_n$ has to intersect the \textit{brown line}. Since the area bounded by $c_n$ and the \textit{brown line} approaches to $0$, the geodesic along $a_{2,n}$ has to be on the right hand side of the \textit{brown line} when $n$ is even larger enough. Therefore $a_{1,n} + c_n + d_n + a_{2, n} > N$ by triangular inequality. Let $N \rightarrow \infty$, $$a_{1,n} + c_n + d_n + a_{2, n} \rightarrow \infty.$$

            Since $j^*(l)\left((t_n, s_n)\right) > a_n + c_n + d_n = a_{1,n} + c_n + d_n + a_{2, n}$,
            $$ j^*(l)\left((t_n, s_n)\right) \rightarrow \infty.$$

        \item
            If $\{(t_n, s_n)\} \subset II$, $$-\hat{s} - \frac{\hat{a}}{2} = 0.$$ Therefore
            $$ \alpha_n \rightarrow \frac{\pi}{2} $$
            $$ a_n + (-s_n, \frac{a_n}{2}) \rightarrow 2 \log\left(\sqrt{2}+1\right). $$

            \begin{figure}[h]
                \begin{center}
                    \scalebox{.6}{\epsffile{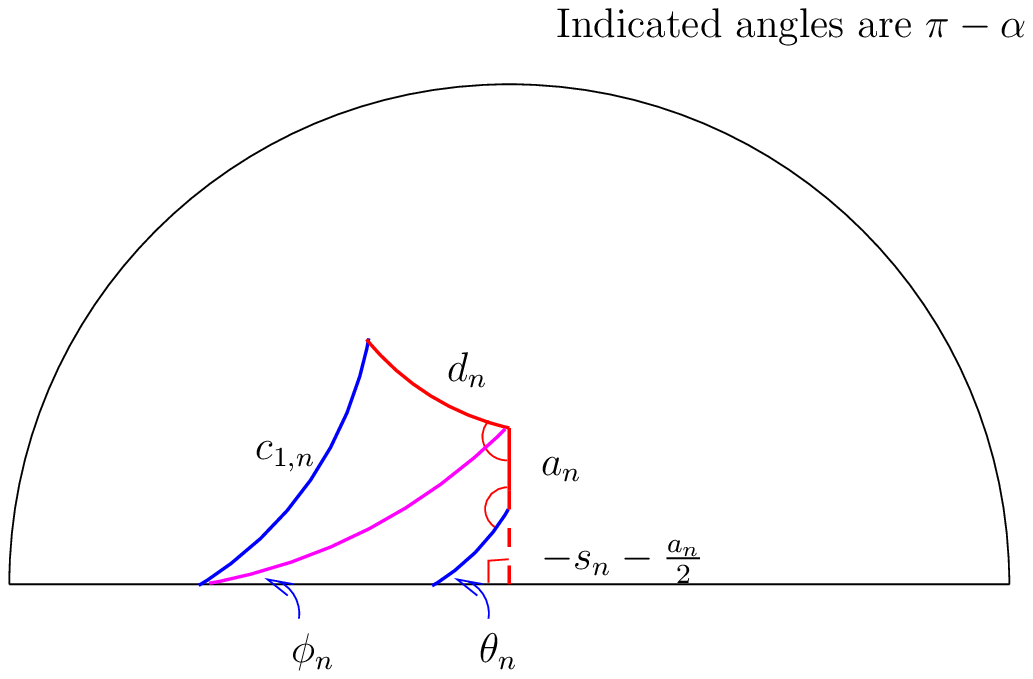}}
                \end{center}
                \caption{$\phi_n \rightarrow 0$ and $a_n + (-s_n, \frac{a_n}{2}) \rightarrow 2 \log\left(\sqrt{2}+1\right)$.}
                \label{fig:PROPER32}
            \end{figure}

            We lift all the geodesic segments $-s_n - \frac{a_n}{2}$'s on the \textit{vertical line} at the origin as in Figure \ref{fig:PROPER32}. Since $\phi_n < \theta_n < \pi - \alpha_n - \frac{\pi}{2}$, $$\phi_n \rightarrow 0.$$

            Consider the right triangle bounded by the \textit{magenta line}, the \textit{red line} and the \textit{horizontal line}. The length of the \textit{magenta line} approaches to infinity by the \textit{Law of Sine}, hence $$c_{1,n} + d_n \rightarrow \infty$$ by triangular inequality. Since  $j^*(l)\left((t_n, s_n)\right) > c_n + d_n > c_{1,n} + d_n$,
            $$ j^*(l)\left((t_n, s_n)\right) \rightarrow \infty.$$

        \item
            If $\{(t_n, s_n)\} \subset III$, See Figure \ref{fig:PROPER33}. The same argument above can be applied to this case as well.
            \begin{figure}[h]
                \begin{center}
                    \scalebox{.6}{\epsffile{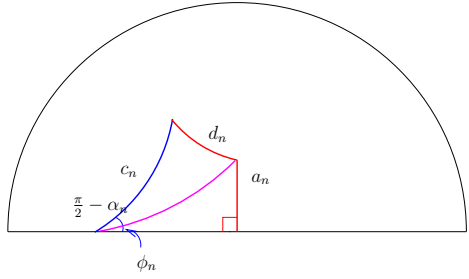}}
                \end{center}
                \caption{$\phi_n \rightarrow 0$ and $a_n \rightarrow 2 \log\left(\sqrt{2}+1\right)$.}
                \label{fig:PROPER33}
            \end{figure}
    \end{itemize}
\end{proof}

This concludes the properness of $j^*(l)$.

\section{The Hyperbolic Structure with Minimal Length}

Let $$\mathcal{A} := \left\{ (t,0) \mid t \in \left(\frac{1}{2}, 1\right)\right\} \subset \mathcal{V}.$$ Since $j^*(l)$ has a unique minimum and it is even on $s$, the minimum is obtained at some $(t_0,0) \in \mathcal{A}$.

In fact, $j^*(l)\mid_{\mathcal{A}}$ is quite explicit.

\begin{thm}
$j^*(l)\mid_{\mathcal{A}} : \left(\frac{1}{2}, 1\right) \rightarrow \mathbb{R}$ is given by
$$ t \mapsto 2\log\left( \frac{\sqrt{5t^2-1}+2t^2}{(2t-1)(1-t)}\right).$$
\label{thm:LENGTH}
\end{thm}

\begin{cor}
    $j^*(l)$ has a unique minimal at $t_0 = \frac{3\sqrt{5}}{10}$. Hence the hyperbolic structure with the minimal $\gamma$ length is given by the hexagon with
    $$ \alpha = \beta = 2 \cos^{-1}\left(\frac{3\sqrt{5}}{10}\right)$$
    $$ a = b = \log\left(\frac{29+12\sqrt{5}}{11}\right)$$
    $$ c = d = \log\left(\frac{21+8\sqrt{5}}{11}\right).$$
\end{cor}

We need the following Lemmas from \cite{Ra}.
\begin{lemma}
    Let $Q$ be a hyperbolic convex quadrilateral with two adjacent right angles, opposite angles $\alpha$, $\beta$, and sides of length $c$, $d$ between $\alpha$, $\beta$ and the right angles, respectively. Then
    $$ \cosh(c) = \frac{\cos(\alpha)\cos(\beta) + \cosh(d)}{\sin(\alpha) \sin(\beta)}.$$
\end{lemma}

\begin{lemma}
    Let $Q$ be a hyperbolic convex quadrilateral with three right angles and fourth angle $\gamma$, and let $a$, $b$ the lengths of sides opposite the angle $\gamma$. Then
    $$ \cos(\gamma) = \sinh(a) \sinh(b).$$
\end{lemma}

\begin{figure}[h]
    \begin{center}
        \scalebox{.6}{\epsffile{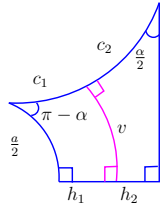}}
    \end{center}
    \caption{A quarter of a hexagon for $t \in \left(\frac{1}{2}, \frac{\sqrt{2}}{2}\right)$.}
    \label{fig:QUARTER}
\end{figure}

\begin{proof}[Proof of Theorem \ref{thm:LENGTH}]
    We first want to find the formula for $t \in \left(\frac{1}{2}, \frac{\sqrt{2}}{2}\right)$. From Figure \ref{fig:QUARTER} and use the above lemmas,
    \begin{eqnarray}
        \cosh(v) & = & \cosh\left(\frac{a}{2}\right) \sin(\pi-\alpha) \notag \\
        \cosh(h_1) & = & \cosh(c_1) \sin(\pi-\alpha) \notag \\
        \cos(\pi-\alpha) & = & \sinh(h_1) \sinh(v) \notag\\
        \cosh(h_2) & = & \cosh(c_2) \sin\left(\frac{\alpha}{2}\right) \notag \\
        \cos\left(\frac{\alpha}{2}\right) & = & \sinh(h_2) \sinh(v). \notag
    \end{eqnarray}
    Then
    \begin{eqnarray}
        \cos^2(\alpha) & = & \left(\cosh^2(c_1) \sin^2(\alpha) - 1 \right)\left(\cosh^2\left(\frac{a}{2}\right) \sin^2(\alpha)-1\right) \notag\\
        \cos^2\left(\frac{\alpha}{2}\right) & = & \left(\cosh^2(c_2) \sin^2\left(\frac{\alpha}{2}\right) - 1 \right)\left(\cosh^2\left(\frac{a}{2}\right) \sin^2(\alpha)-1\right). \notag\\
    \end{eqnarray}
    Use (\ref{eqn:a})
    \begin{eqnarray}
        \cosh^2(c_1) & = & \frac{t^2}{(4t^2-1)(1-t^2)}\notag\\
        \cosh^2(c_2) & = & \frac{5t^2-1}{(4t^2-1)(1-t^2)}.
    \end{eqnarray}
    Hence
    \begin{eqnarray}
        c_1 & = & \log\left( \frac{(2t+1)(1-t)}{\sqrt{(4t^2-1)(1-t^2)}}\right) \notag\\
        c_2 & = & \log\left( \frac{\sqrt{5t^2-1}+2t^2}{\sqrt{(4t^2-1)(1-t^2)}}\right).
    \end{eqnarray}
    Then
    \begin{equation}
        c = c_1 + c_2 = \log\left( \frac{\sqrt{5t^2-1}+2t^2}{(2t-1)(1+t)}\right).
        \label{eqn:c}
    \end{equation}
    It can be shown that (\ref{eqn:c}) works for all $t \in \left(\frac{1}{2}, 1\right)$.
    Therefore
    \begin{equation}
        j^*(l) = a + b + c + d = 2a+2c = 2\log\left( \frac{\sqrt{5t^2-1}+2t^2}{(2t-1)(1-t)}\right).
    \end{equation}
\end{proof}

\end{document}